\documentclass[12pt, twoside]{article}
\usepackage{amsmath,amsthm,amssymb}
\usepackage{times}
\usepackage{enumerate}
\usepackage{bm}
\usepackage[all]{xy}
\usepackage{mathrsfs}
\usepackage{amscd}
\usepackage{color}
\usepackage{comment}
\def\titlerunning#1{\gdef\titrun{#1}}
\makeatletter
\def\author#1{\gdef\autrun{\def\and{\unskip, }#1}\gdef\@author{#1}}
\def\address#1{{\def\and{\\\hspace*{18pt}}\renewcommand{\thefootnote}{}%
		\footnote {#1}}%
	\markboth{\autrun}{\titrun}}
\makeatother
\def\email#1{e-mail: #1}
\def\subjclass#1{{\renewcommand{\thefootnote}{}%
		\footnote{\emph{Mathematics Subject Classification (2020):} #1}}}
\def\keywords#1{\par\medskip
	\noindent\textbf{Keywords.} #1}

\newtheorem{theorem}{Theorem}[section]
\newtheorem{corollary}[theorem]{Corollary}
\newtheorem{lemma}[theorem]{Lemma}
\newtheorem{proposition}[theorem]{Proposition}
\theoremstyle{definition}
\newtheorem{definition}[theorem]{Definition}
\newtheorem{remark}[theorem]{Remark}

\numberwithin{equation}{section}
\newtheorem{conjecture}{Conjecture}

\xyoption{all}

\def \a {\alpha }
\def \b {\beta}

\def \de {\delta}
\def \De {\Delta}
\def \la {\lambda}
\def \La {\Lambda}
\def\w {\omega}
\def\Om{\Omega}

\def\na {\nabla}

\begin{document}
	\baselineskip=17pt
	
	\titlerunning{Decomposition of the Curvature Operator and Applications to the Hopf Conjecture }
	\title{Decomposition of the Curvature Operator and Applications to the Hopf Conjecture }
	
\author{Teng Huang and Weiwei Wang}

\date{}

\maketitle

\address{Teng Huang: School of Mathematical Sciences, CAS Key Laboratory of Wu Wen-Tsun Mathematics, University of Science and Technology of China, Hefei, Anhui, 230026, People’s Republic of China; \email{htmath@ustc.edu.cn;htustc@gmail.com}}
\address{Weiwei Wang: School of Mathematical Sciences, University of Electronic Science and Technology of China, Chengdu, Sichuan, 611731, People’s Republic of China; \email{wawnwg123@163.com}}

	\begin{abstract}
In this article, we investigate the interplay between the curvature operator, Weyl curvature, and the Hopf conjecture on compact Riemannian manifolds of even dimension. By decomposing the curvature operator into Hermitian components, we develop eigenvalue criteria for sectional curvature and prove vanishing theorems for Betti numbers under integral bounds on the Weyl tensor.  Our results confirm the Hopf conjecture for manifolds with sufficiently small Weyl curvature, including locally conformally flat cases, and provide new rigidity theorems under harmonic Weyl curvature conditions.
\end{abstract}
\section{Introduction}
One of the most famous conjectured relationships between sectional curvature and topological invariants was formulated by H. Hopf in the 1930s. Hopf posed the fundamental question of whether the sign of the sectional curvature governs the Euler characteristic of a compact Riemannian manifold.
\begin{conjecture}(Hopf Conjecture)
Let $(M,g)$ be a compact Riemannian manifold of even dimension $2n$ ($n \geq 2$) with sectional curvature $\sec_g$. Then the following implications hold:
\begin{equation*}
\begin{split}
(-1)^{n}\chi(M)\geq0\ & if\ \sec_{g}\leq0,\\
(-1)^{n}\chi(M)>0\ & if\ \sec_{g}<0,\\
\chi(M)=0\ & if\ \sec_{g}=0,\\
\chi(M)\geq0\  & if\ \sec_{g}\geq0,\\
\chi(M)>0\ & if\ \sec_{g}>0.\\
\end{split}
\end{equation*}	
\end{conjecture}
The conjecture holds for dimensions $2$ and $4$, as the Gauss–Bonnet integrands in these low-dimensional cases exhibit the required sign behavior (see Bishop and Goldberg \cite{BG64} or Chern \cite{Che55}). However, in higher dimensions, it has been established that the sectional curvature's sign does not universally determine the sign of the Gauss-Bonnet-Chern integrand \cite{Ger76}. In the 1990s, K. Grove initiated a research program to investigate the positive sectional curvature case of this conjecture. His approach centered on examining positively curved metrics admitting large isometry groups. This strategy has yielded significant results (see, e.g., \cite{FR04,Ken13,RS05,Wil03}) and remains an active area of research in Riemannian geometry. The conjecture was established by Kennard, Wiemeler, and Wilking under the stronger hypothesis that the isometry group has rank at least five \cite{KWW21}. Subsequently, Nienhaus \cite{Nie22} weakened this assumption, proving the result under the existence of an isometric $T^{4}$-action, thereby improving upon the earlier work requiring a $T^{5}$-action.

For the case of negative sectional curvature, a vanishing theorem was established in \cite{JX00}, showing that the space of $L^{2}$ $k$-forms vanishes for $k$ in a specific range determined by the curvature's pinching constants. Under sufficiently strong pinching conditions, this result provides an affirmative answer to Hopf's question in this setting. Their work improves upon earlier results by Donnelly and Xavier \cite{DX84}. The Hopf Conjecture under pinched sectional curvature is also addressed in \cite{BK78}. The use of $L^{2}$-cohomology as a tool to study this conjecture was initially proposed by Dodziuk \cite{Dod79} and Singer \cite{Sin77}. Gromov \cite{Gro91} later advanced this program in the K\"{a}hler setting, proving the conjecture for K\"{a}hler manifolds that are homotopy equivalent to closed manifolds with strictly negative sectional curvature. To further address the Hopf conjecture for K\"{a}hler manifolds with nonpositive sectional curvature, Cao–Xavier \cite{CX01} and Jost–Zuo \cite{JZ00} independently extended Gromov’s approach by introducing the notion of K\"{a}hler parabolicity. This class includes all closed nonpositively curved K\"{a}hler manifolds, and they demonstrated that such manifolds indeed satisfy the expected sign conditions on their Euler characteristics.

In this article, we investigate the Hopf conjecture for compact Riemannian manifolds under the assumption that the Weyl curvature is sufficiently small in an appropriate sense. Our approach begins by establishing fundamental relationships between the eigenvalues of the curvature operator, sectional curvatures, and Weyl curvature.

For the Riemannain curvature tensor $Rm$, there is a decomposition formula (cf. \cite{Cha06,Pet16}):
\begin{equation*}
Rm=S_{g}\odot g+\mathcal{W}_{g},
\end{equation*}
where $\mathcal{W}_{g}$ denotes the Weyl curvature tensor of the metric $g$, $S_{g}$ represents the Schouten tensor, defined by
\begin{equation*}
S_{g} = \frac{1}{n-2}\left(Ric_{g} - \frac{R_{g}}{2(n-1)} \cdot g\right),
\end{equation*}
with $\text{Ric}_{g}$ and $R_{g}$ being the Ricci tensor and scalar curvature of $g$, respectively. The associated algebraic curvature operator $\mathfrak{R}: \Lambda^{2}TM \rightarrow \Lambda^{2}TM$ is defined via the relation:
$$g(\mathfrak{R}(X\wedge Y),Z\wedge W)=Rm(X,Y,Z,W).$$

We introduce two auxiliary curvature operators $\mathfrak{S}$ and $\mathfrak{W}$, corresponding to $S_{g} \odot g$ and $\mathcal{W}_{g}$ respectively, defined by:
$$g(\mathfrak{S}(X\wedge Y),Z\wedge W)=(S_{g}\odot g)(X,Y,Z,W)$$ and$$ g(\mathfrak{W}(X\wedge Y),Z\wedge W)=\mathcal{W}_{g}(X,Y,Z,W).$$
\begin{theorem}\label{T3}
	Let $(M,g)$ be a compact manifold of dimension $n$, $n\geq4$, and let $\{e_{i}\}_{1,\cdots,n}$ be an orthonormal eigenbasis for the Ricci curvature tensor. Then the curvature operator $\mathfrak{R}$ satisfies:
	\begin{equation*}
	\begin{split}
	&\langle\mathfrak{R}(e_{i}\wedge e_{j}),(e_{i}\wedge e_{j})\rangle=\la_{ij}+\mathcal{W}_{ijij}=Rm(e_{i},e_{j},e_{i},e_{j}),\\
	&\langle\mathfrak{R}(e_{i}\wedge e_{j}),(e_{k}\wedge e_{l})\rangle=\mathcal{W}_{ijkl},\forall (i,j)\neq (k,l),\\
	\end{split}
	\end{equation*}
	where $\lambda_{ij}=\frac{1}{n-2}\left(\lambda_{i} + \lambda_{j} - \frac{1}{n-1}R_{g}\right)$ and $i<j$, $k<l$. Consequently, we have the decomposition $$\mathfrak{R} = \mathfrak{S} + \mathfrak{W},$$
		where
		$$\mathfrak{S}=\left(\begin{array}{ccc}\la_{12} & \cdots & 0 \\ \vdots & \ddots & \vdots \\0 & \cdots & \la_{(n-1)n}\end{array}\right)$$
		and
		$$\mathfrak{W}=\left(\begin{array}{ccc}\mathcal{W}_{1212} &\cdots  & \mathcal{W}_{ijkl} \\ \vdots &\ddots & \vdots \\ \mathcal{W}_{ijkl}  & \cdots & \mathcal{W}_{(n-1)n(n-1)n}\end{array}\right).$$
\end{theorem}
Building upon Theorem \ref{T3}, we establish the validity of the Hopf conjecture for compact manifolds with suitably controlled Weyl curvature. The simplest case is naturally when the Weyl tensor is zero, i.e., the manifold is conformally locally flat. We observe that in this case, the curvature operator can be diagonalized, and its eigenvalues are controlled by the sectional curvature of the manifold (see Corollary \ref{C5} and Proposition \ref{P1}). Therefore, we can immediately confirm the Hopf conjecture when the Weyl curvature vanishes.
\begin{theorem}(=Theorem \ref{T4})
	The Hopf conjecture holds for all compact locally conformally flat manifolds of even dimension $2n$ ($n \geq 2$).
\end{theorem}
When the Weyl curvature is dominated by the sectional curvature, we can also establish the positivity of the curvature operator. Hence we obtain the following result.
\begin{theorem}\label{T2}
	Let $(M,g)$ be a compact Riemannian manifold of even dimension $2n$ ($n \geq 2$) whose sectional curvature satisfies $\sec_g \geq a> 0$ (resp. $\sec_g \leq -a <0$). If the Weyl curvature $\mathcal{W}_g$ satisfies
$$|\mathcal{W}_{g}|<\sqrt{\frac{n(n-1)}{(n+1)(n-2)}}a,$$
then the curvature operator $\mathfrak{R}$ is positive (resp. negative). In particular,  the Euler characteristic of $M$ satisfies
$$\chi(M) > 0\ (resp. (-1)^n\chi(M)>0).$$
\end{theorem}

When relaxing the pointwise boundedness condition on the Weyl curvature to an integral bound, the curvature operator may fail to maintain positivity. Nevertheless, through application of the Sobolev inequality, we can still establish the some vanishing theorems. On an $n$-dimensional manifold, positive $k$th-intermediate Ricci curvature is a condition interpolating between positive sectional curvature $(k=1)$ and positive Ricci curvature $(k=n-1)$. There has recently been increased interest in studying manifolds with lower bounds on $k$th-intermediate Ricci curvature  (cf.\cite{KM18,Mou22,RW23,RS05}). The precise definition of the $k$-th intermediate Ricci curvature can be found in Definition \ref{D2} of Section 2.4.

\begin{theorem}\label{T6}
Let $(M,g)$ be a compact $n$-dimensional manifold ($n \geq 3$) with positive $k$th-intermediate Ricci curvature for some $1 \leq k \leq n-1$, satisfying
	$$Ric_{k}\geq ka>0.$$ 
	Suppose that the Weyl curvature $\mathcal{W}_{g}$ satisfies
$$\bigg{(}\frac{\int_{M}|\mathcal{W}_{g}|^{\frac{n}{2}}}{{\rm{Vol}}(g)}\bigg{)}^{\frac{2}{n}}<\sqrt{\frac{kN}{N-k}}\frac{n-2}{n}a,$$	
where $N=\frac{n(n-1)}{2}$. Then
	\begin{itemize}
		\item[(1)] For $1\leq k\leq \lceil\frac{n}{2}\rceil$, we have $b_{p}(M)=0$ for all $1\leq p\leq n-1$.
	\item[(2)] For $ \lceil\frac{n}{2}\rceil+1\leq k\leq n-1$, we have
	$b_{1}(M)=\cdots=b_{n-k}(M)$.
\end{itemize}
\end{theorem}
\begin{remark}(1)	The $L^{\frac{n}{2}}$-norm of the Weyl tensor is conformally invariant, remaining unchanged under any homothety transformation $g \mapsto cg$ (where $c > 0$ is constant).\\
(2) The normalized volume term $[a{\rm{Vol}}(g)]^{\frac{2}{n} }$ admits a uniform upper bound given by $[{\rm{Vol}}(S^{n})]^{\frac{2}{n}}$, where $S^{n}$ denotes the standard $n$-sphere. Specifically, the following  inequalities holds:
$$[a{\rm{Vol}}(g)]^{\frac{2}{n} }\leq \frac{\la(g) }{n(n-1)}\leq[{\rm{Vol}}(S^{n})]^{\frac{2}{n}}$$
where $\la(g)$ is the Yamabe constant (see Proposition \ref{P6}).
\end{remark}
We now proceed to consider the case of manifolds with positive Yamabe constant. First, we decompose the curvature operator in an alternative form, through which we obtain a lower bound estimate for its eigenvalues that depends on the scalar curvature (see Proposition \ref{P8}). Based on this eigenvalue estimate, we derive the following results.
\begin{theorem}\label{T11}
	Let $(M,g)$ be a compact $n$-dimensional manifold, ($n \geq 8$), with positive Yamabe constant $\la(g)$ and $ 1\leq k\leq n-1$. If 
	$$ \bigg{(}\int_{M}(\sqrt{\frac{ 1}{n-2}}|\overset{\circ}{Ric}_{g}|+|\mathcal{W}_{g}|)^{ \frac{n}{2} }\bigg{)}^{\frac{2}{n} }<\frac{\la(g)}{n(n-1)}\sqrt{\frac{Nk}{N-k} },$$
	where $N=\frac{n(n-1)}{2}$, then 
		\begin{itemize}
		\item[(1)] For $1\leq k\leq \lceil\frac{n}{2}\rceil$, we have $b_{p}(M)=0$ for all $1\leq p\leq \lceil\frac{n}{2}\rceil-2$.
		\item[(2)] For $ \lceil\frac{n}{2}\rceil+1\leq k\leq n-1$, we have
		$b_{p}(M)=0$ for all $1\leq p\leq \min\{\lceil\frac{n}{2}\rceil-2,n-k \}$.
	\end{itemize}
\end{theorem}
Hebey and Vaugon \cite{HV97} have proved the existence of a dimensional constant such that if the Yamabe metric $g$ of a conformal class which is either locally conformally flat or conformally Einstein satisfies
$$\bigg{(}\int_{M}|Z_{g}|^{\frac{n}{2}}\bigg{)}^{\frac{2}{n} }<c(n)R_{g},$$
then $(M,g)$ is isometric to a quotient of the standard sphere.  Here $Z_{g}:=\frac{1}{n-2}\overset{\circ}{Ric}_{g}\odot g+\mathcal{W}_{g}$ is the concircular curvature tensor (cf.\cite{HV97})

	Building upon Theorem \ref{T11}, we demonstrate that for the manifolds equipped with metrics satisfying appropriate curvature conditions (without imposing any additional assumptions on the conformal class), the Betti numbers will coincide with those of the $n$-dimensional sphere.
	\begin{corollary}\label{C3}
	Let $(M,g)$ be a compact $n$-dimensional manifold, $n \geq 4$, with positive scalar curvature $R_{g}>0$ and $ 1\leq k\leq n-1$. If 
	$$ \bigg{(}\int_{M}(\sqrt{\frac{ 1}{n-2}}|\overset{\circ}{Ric}_{g}|+|\mathcal{W}_{g}|)^{ \frac{n}{2} }\bigg{)}^{\frac{2}{n} }<\frac{(n-2)}{n^{2}(n-1)}\sqrt{\frac{k(N-k)}{N} }\la(g),$$
	where $N=\frac{n(n-1)}{2}$ and $\la(g)$ is the Yamabe constant, then

	\begin{itemize}
		\item[(1)] For $1\leq k\leq \lceil\frac{n}{2}\rceil+1$, we have $b_{p}(M)=0$ for all $1\leq p\leq n-1$.
		\item[(2)] For $ \lceil\frac{n}{2}\rceil+2\leq k\leq n-1$, we have
		$b_{p}(M)=0$ for all $1\leq p\leq n-k$.
	\end{itemize}
\end{corollary}
\begin{remark}
The  concircular curvature tensor satisfies $$|Z_{g}|^{2}=\frac{4}{n-2}|\overset{\circ}{Ric}_{g}|^{2}+|\mathcal{W}_{g}|^{2}.$$
Therefore,
$$|Z_{g}|=\sqrt{\frac{4}{n-2}|\overset{\circ}{Ric}_{g}|^{2}+|\mathcal{W}_{g}|^{2} }\geq\frac{2}{\sqrt{5}}(\sqrt{\frac{ 1}{n-2}}|\overset{\circ}{Ric}_{g}|+|\mathcal{W}_{g}|).$$
This inequality follows from the Cauchy-Schwarz type estimate
$|\a a+\b b|\leq\sqrt{\a^{2}+\b^{2} }\sqrt{a^{2}+b^{2}}$
where we specifically take:
$$
a= \sqrt{\frac{4}{n-2}}|\overset{\circ}{Ric}_{g}|, \ b=|\mathcal{W}_{g}|, \ \alpha=\frac{1}{2}, \ \beta= 1.$$

\end{remark}
For Einstein manifolds, Tachibana \cite{Tac74} established that nonnegative curvature operator implies local symmetry. Brendle \cite{Bre10} later strengthened this result by requiring only nonnegative isotropic curvature. These curvature conditions involve pointwise pinching of $\mathcal{W}_{g}$ relative to the scalar curvature. Parallel developments have studied integral pinching conditions on $\mathcal{W}_{g}$, particularly through its $L^{\frac{n}{2}}$-norm, as seen in \cite{Cat16,IS02,Sin92}.

	A natural generalization of Einstein metrics is the class of metrics with harmonic curvature, characterized by divergence-free Weyl tensor (i.e., $\text{div}\mathcal{W}_{g} = 0$) and constant scalar curvature. This class is particularly significant due to its connections with Yang-Mills theory, and has been extensively studied in works such as \cite{Der77,Fu17,HV97,IS02,Kim11,Tru68}. Notably, rigidity results under integral pinching conditions were obtained in \cite{Fu17,HV97}, while Tran \cite{Tra17} established both pointwise and integral rigidity/gap theorems for closed manifolds with harmonic Weyl curvature in dimensions higher than four. These results include a generalization of Tachibana's theorem for nonnegative curvature operators, which depends on new Bochner-Weitzenb\"{o}ck-Lichnerowicz type formulas for the Weyl tensor that extend the identities in dimension four. 
	
	Through the decompositions of the curvature operator, we also obtain some gap results related to the Weyl curvature.
	
\begin{theorem}\label{T1}
	Let $(M,g)$ be a compact $n$-dimensional Riemannian manifold, ($n \geq 4$), with harmonic Weyl curvature $\mathcal{W}_{g}$.
	\begin{itemize}
		\item [(1)] If $g$ has $k$th-intermediate
		Ricci curvature for some $1 \leq k \leq \lfloor\frac{n-1}{2}\rfloor$, i.e., $Ric_{k}\geq ka>0$, then 
	$$\left(\frac{\int_{M}|\mathcal{W}_{g}|^{\frac{n}{2}}}{{\rm{Vol}}(g)}\right)^{\frac{2}{n}}\geq\min\left\{\sqrt{\frac{Nk}{N-k}}a,\frac{n(n-2)}{8(n-1)}\sqrt{\frac{Nk}{N-k}}a \right\},$$
where $N=\frac{n(n-1)}{2}$, unless $\mathcal{W}_{g}=0$.			
		\item [(2)] If the Yambabe constant $\la(g)$ is positive and $n\geq8$, then 	$$ \left(\int_{M}\left(\sqrt{\frac{ 1}{n-2}}|\overset{\circ}{Ric}_{g}|+|\mathcal{W}_{g}|\right)^{\frac{n}{2}}\right)^{\frac{2}{n} }\geq \frac{1}{\sqrt{2n}(n-1)}\la(g),$$
		unless $\mathcal{W}_{g}=0$.
	\end{itemize} 
\end{theorem}

\begin{remark}
In \cite{Tra17}, the author proved that for a closed Rimennian $n$-dimensional, $n\geq6$, manifold $(M,g)$ with non-zero harmonic Weyl tenor and positive Yamabe constant $\la(g)$, there exist positive constants $c_{1}(n), c_{2}(n)$ such that
$$a_{1}(n)\|\mathcal{W}_{g}\|_{L^{\frac{n}{2} }}+a_{n}\|\overset{\circ}{Ric}_{g}\|_{L^{\frac{n}{2} }}\geq\frac{n-2}{4(n-1)}\la(g),$$
where $\overset{\circ}{Ric}_{g}=Ric_{g}-\frac{R_{g}}{n}g$ is the trace-flat Ricci tensor.

Our results here do not incorporate the trace-free part of the Ricci curvature tensor. But the curvature assumption of case (I) in Theorem \ref{T1} only implies that $\la(g)\geq n(n-1)a[{\rm{Vol}}(g)]^{-\frac{2}{n}}$, we cannot conclude that under this curvature condition, the $L^{\frac{n}{2}}$-norm of the Weyl curvature can be controlled by the Yamabe constant. 
\end{remark}

\section{Curvature operator }
\subsection{Preliminaries }
Let $(V,g)$ be an $n$-dimensional Euclidean vector space. The metric $g$ induces a metric on $\otimes^{k}V^{\ast}$ and $\La^{k}V$. In particular, if $\{e_{i}\}_{i=1,\cdots,n}$ is an orthonormal basis for $V$ , then  $\{e_{i_{1}}\wedge\cdots\wedge e_{i_{k}} \}$ is an orthonormal basis of $\La^{k}V$. The derivative of the regular representation of $O(n)$ on $(V,g)$ induces a derivation on tensors (see \cite[Section 1.2]{PW21a}): If $T\in\otimes^{k}V^{\ast}$ and $L\in\mathfrak{so}(V)$, then
$$(LT)(X_{1},\cdots,X_{k})=-\sum_{i=1}^{k}T(X_{1},\cdots,LX_{i},\cdots,X_{k}).$$
Notice that  $\La^{2}$ inherits a Lie algebra structure from $\mathfrak{so}(V)$.
\begin{definition}(cf.\cite[Definition 1.2]{PW21a} or \cite{PW21b})\label{D1}
Let $\mathfrak{g}$ be a Lie subalgebra of $\mathfrak{so}(V)$. For $T\in\otimes^{k}V^{\ast}$ define $T^{\mathfrak{g}}\in\mathfrak{g}\otimes(\otimes^{k}V^{\ast})$ implicitly by
$$g(L,T^{\mathfrak{g}}(X_{1},\cdots,X_{k}))=(LT)(X_{1},\cdots,X_{k})$$
for all $L\in\mathfrak{g}\subset \mathfrak{so}(V)=\La^{2}V$.
\end{definition}
The following norms and inner products for tensors, whose components are with respect to an arbitrary choice of an orthonormal basis, will be used throughout:\\
For $T\in\otimes^{k}V^{\ast}$, define
$$|T|^{2}=\sum_{i_{1},\cdots,i_{k}}(T_{i_{1}\cdots i_{k} })^{2}.$$
For a $p$-form $\w\in\La^{p}V^{\ast}$, define
$$|\w|^{2}=\sum_{i_{1},\cdots,i_{p}}(\w_{i_{1}\cdots i_{p} })^{2}.$$
Let $$Sym^{2}(V)\subset\otimes^{2}V^{\ast}$$
denote the space of symmetric $(0,2)$-tensors on $V$.
\begin{definition}
	The Kulkarni-Nomizu product of $S,T \in Sym^{2}(V)$ is given by
	\begin{equation*}
	\begin{split}
	(S\odot T )(X,Y,Z,W)&=S(X,Z)T(Y,W)-S(X,W)T(Y,Z)\\
	&+S(Y,W)T(X,Z)-S(Y,Z)T(X,W).\\
	\end{split}
	\end{equation*}
In particular, the metric tensor $g$ satisfies
	$$(g\odot g )(X,Y,Z,W)=2\big{\{}g(X,Z)g(Y,W)-g(X,W)g(Y,Z)\big{\}}.$$
\end{definition}

\subsection{A decomposition of the curvature operator}
Let $(M,g)$ be an oriented, compact Riemannian manifold of dimension $n\geq3$. The Riemann curvature tensor $Rm$ is defined by
$$Rm(X,Y)Z=\na_{X}\na_{Y}Z-\na_{Y}\na_{X}Z+\na_{[X,Y]}Z$$
where $\na$ is the Levi-Civita connection. For a $(0,k)$-tensor $T$ set
$$\mathfrak{Ric}(T)(X_{1},\cdots,X_{k})=\sum_{i=1}^{k}\sum_{j=1}^{n}(Rm(X_{i},e_{j})T)(X_{1},\cdots,e_{j},\cdots,X_{k}),$$
where $\{e_{1},\cdots,e_{n}\}$ is a local orthonormal frame (cf.\cite{PW21a}).

The Schouten tensor of $g$, denote by $S_{g}$, is given by
\begin{equation}
S_{g}=\frac{1}{n-2}(Ric_{g}-\frac{R_{g}}{2(n-1)}\cdot g),
\end{equation}
where $Ric_{g}$ and $R_{g}$ are the Ricci tensor and scalar curvature of $g$ respectively. The Riemann curvature tensor admits a decomposition into its Weyl tensor $\mathcal{W}_{g}$ and a term involving the Schouten tensor:
\begin{equation}\label{E1}
Rm=S_{g}\odot g+\mathcal{W}_{g},
\end{equation}
where $\odot$ denotes the Kulkarni-Nomizu product. 

The associated algebraic curvature operator $\mathfrak{R}:\La^{2}TM\rightarrow\La^{2}T(M)$ is defined via
$$g(\mathfrak{R}(X\wedge Y),Z\wedge W)=Rm(X,Y,Z,W).$$
Special assumptions on $\mathfrak{R}$ also have geometric consequences, for example:
\begin{itemize}
\item[(i)] a compact, connected and oriented Riemannian manifold with positive curvature operator $\mathfrak{R}$ has the rational homology of the sphere (cf. \cite{Mey71});
\item[(ii)] the positivity of $\mathfrak{R}$ implies that the Gauss-Bonnet-Chern integrand is positive (cf. \cite{Kul72}).
\end{itemize}
The Gauss-Bonnet-Chern formula for $2n$-dimensional manifolds reveals a fundamental connection between curvature eigenvalues and Euler characteristic (cf.\cite{BK78,Kul72}). We then have
\begin{proposition}\label{P2}
	Let $(M,g)$ be a compact $2n$-dimensional Riemannian manifold. 
	\begin{itemize}
		\item[(i)]
		If curvature operator $\mathfrak{R}$ is nonnegative, then the Euler characteristic of $M$ obeys $\chi(M)>0$. Furthermore, if  $\mathfrak{R}$ is positive, then the Betti numbers of $M$ are rational (co)homology spheres.
		\item[(ii)] If  curvature operator $\mathfrak{R}$ is nonpositive, then the Euler characteristic of $M$ obeys $(-1)^{n}\chi(M)\geq0$. Furthermore, if  $\mathfrak{R}$ is negative, then $(-1)^{n}\chi(M)>0$. 
	\end{itemize}
\end{proposition}
	\begin{lemma}\label{P5}
		Let $(M,g)$ be a compact manifold of dimension $n$, $n\geq4$, and let $\{e_{i}\}_{1,\cdots,n}$, be an orthonormal eigenbasis of the Ricci curvature with corresponding eigenvalues $\{\la_{i} \}_{1,\cdots,n}$. Then the curvature operator $\mathfrak{S}:\La^{2}TM\rightarrow\La^{2}TM$ corresponding to $S_{g}\odot g$ acts on $2$-forms as
		\begin{equation*}
		\mathfrak{S}(e_{i}\wedge e_{j})=\frac{1}{n-2}\left(\la_{i}+\la_{j}-\frac{R_{g}}{n-1}\right)(e_{i}\wedge e_{j}).
		\end{equation*}
		It implies that $\{e_{i}\wedge e_{j}\}_{1\leq i<j\leq n}$ is an  an eigenbasis of $\mathfrak{S}$ with eigenvalues
		$\la_{ij}:=\frac{1}{n-2}(\la_{i}+\la_{j}-\frac{R_{g}}{n-1})$.
	\end{lemma}
	\begin{proof}
		The $\{e_{i}\}_{i=1,\cdots,n}$ is an orthonormal eigenbasis of Ricci curvature tensor $Ric_{g}$, which satisfies
		\begin{equation*}
		Ric_{g}(e_{i})=\la_{i}e_{i},\ \quad g(e_{i},e_{j})=\de_{ij}.
		\end{equation*}
		By the definition of  Kulkarni-Nomizu product, we compute
		\begin{equation}\label{Ricg1}
		\begin{split}
		(Ric\odot g)(e_{i},e_{j},e_{k},e_{l})&=Ric_{g}(e_{i},e_{k})g(e_{j},e_{l})-Ric_{g}(e_{i},e_{l})g(e_{j},e_{k})\\
		&+Ric_{g}(e_{j},e_{l})g(e_{i},e_{k})-Ric_{g}(e_{j},e_{k})g(e_{i},e_{l})\\
		&=\la_{i}\de_{ik}\de_{jl}-\la_{i}\de_{il}\de_{jk}+\la_{j}\de_{jl}\de_{ik}-\la_{j}\de_{jk}\de_{il}\\
		&=(\la_{i}+\la_{j})(\de_{ik}\de_{jl}-\de_{il}\de_{jk}).\\
		\end{split}
		\end{equation}
		Substituting \eqref{Ricg1} into $S_{g}\odot g$, we obtain
		\begin{equation*}
		\begin{split}
		(S_{g}\odot g)(e_{i},e_{j},e_{k},e_{l})&=\left[\frac{1}{n-2}\left(Ric_{g}-\frac{R_{g}}{2(n-1)}\cdot g\right)\odot g\right](e_{i},e_{j},e_{k},e_{l})\\
		&=\frac{1}{n-2}\left(\la_{i}+\la_{j}-\frac{R_{g}}{n-1}\right)(\de_{ik}\de_{jl}-\de_{il}\de_{jk}).
		\end{split}
		\end{equation*}	
		For the curvature operator $\mathfrak{S}:\La^{2}TM\rightarrow\La^{2}TM$ associated to $S_{g} \odot g$ and $\w=\sum_{i,j}\w_{ij}e_i\wedge e_j$,  the $e_{k}\wedge e_{l}(k<l)$ component of $\mathfrak{S}(\w)$  is
		$$ (\mathfrak{S}(\w))_{kl}=\sum_{i,j}(S_{g}\odot g)(e_{i},e_{j},e_{k},e_{l})\w_{ij}.$$
		Thus
		$$\mathfrak{S}(e_{i}\wedge e_{j})_{kl}=\frac{1}{n-2}(\la_{i}+\la_{j}-\frac{R_{g}}{n-1})(\de_{ik}\de_{jl}-\de_{il}\de_{jk}).$$
		\begin{itemize}
			\item if $(i,j)=(k,l)$, then  
			$$\mathfrak{S}(e_{i}\wedge e_{j})_{ij}=\frac{1}{n-2}(\la_{i}+\la_{j}-\frac{R_{g}}{n-1}).$$
			\item if $(i,j)\neq (k,l)$, then
			$$\mathfrak{S}(e_{i}\wedge e_{j})_{kl}=0.$$
		\end{itemize}
		Since  
		$\{e_{i}\wedge e_{j}\}_{1\leq i<j\leq n}$ forms an orthonormal basis for $\La^{2}TM$, it follows that this set is an eigenbasis of $\mathfrak{S}$ with eigenvalues
		$\la_{ij}:=\frac{1}{n-2}(\la_{i}+\la_{j}-\frac{R_{g}}{n-1})$, which proves the lemma.
	\end{proof}
	\begin{proof}[\textbf{Proof of Theorem \ref{T3}}]
		
		Using Lemma \ref{P5}, we obtain the following relations for the curvature operator $\mathfrak{R}=\mathfrak{S}+\mathfrak{W}$ acting on the orthonormal basis $\{e_{i}\wedge e_{j}\}_{1\leq i<j\leq n}$ of $\Lambda^2 TM$:
		\begin{itemize}
			\item For the diagonal components:
			\begin{equation*}
			\langle \mathfrak{R}(e_i\wedge e_j), e_i\wedge e_j \rangle =\langle \mathfrak{S}(e_i\wedge e_j), e_i\wedge e_j \rangle+\langle \mathfrak{W}(e_i\wedge e_j), e_i\wedge e_j \rangle= \lambda_{ij} + \mathcal{W}_{ijij} = Rm(e_i,e_j,e_i,e_j),
			\end{equation*}
			where $\lambda_{ij} = \frac{1}{n-2}(\lambda_i + \lambda_j - \frac{R_g}{n-1})$ are the eigenvalues of $\mathfrak{S}$ as in Lemma \ref{P5}, and $\mathcal{W}_{ijij}$ are the corresponding components of the Weyl tensor by the definition of $\mathfrak{W}$.
			\item For the off-diagonal components when $(i,j) \neq (k,l)$:
			\begin{equation*}
			\langle \mathfrak{R}(e_i\wedge e_j), e_k\wedge e_l \rangle=\langle \mathfrak{W}(e_i\wedge e_j), e_k\wedge e_l \rangle = \mathcal{W}_{ijkl},
			\end{equation*}
			since $\langle \mathfrak{S}(e_i\wedge e_j), e_k\wedge e_l \rangle=0$. This implies that the off-diagonal terms of the curvature operator are completely determined by the Weyl tensor.	
		\end{itemize}
		This completes the proof of the theorem.	
	\end{proof}

\subsection{Curvature operator of  locally conformally flat manifold }
By the conformal invariance of the Weyl tensor, the study of conformal metric deformations reduces to understanding the Schouten tensor (see \cite{GLW04,GLW05,GVW03,GW04,Gur94}). For locally conformally flat manifolds, the curvature tensor decomposes simply as:
\begin{equation*}
Rm=S_{g}\odot g.
\end{equation*}
\begin{corollary}\label{C5}
	Let $(M,g)$ be a compact locally conformally flat manifold, $n\geq3$. Let $\la_{1}\leq\cdots\la_{n}$ be the eigenvalues of Ricci curvature tensor. Then the eigenvalues of curvature operator are 
	$$\{\la_{ij}:=\frac{1}{n-2}\bigg{(}\la_{i}+\la_{j}-\frac{1}{n-1}R_{g}\bigg{)} \}_{1\leq i<j\leq n}.$$ 
\end{corollary}
\begin{proof}
This follows immediately from Proposition \ref{P5}.
\end{proof}	
\begin{remark}
	While the complete ordering of $\{\la_{ij}\}$ is unavailable, we can identify:
	\begin{itemize}
		\item Minimal two eigenvalues: $\la_{12}<\la_{13}$;
		\item Maximal two eigenvalues: $\la_{(n-2)n}<\la_{(n-1)n}$.
	\end{itemize}		
\end{remark}
\begin{proposition}(cf.\cite{Pet16})\label{P1}
	Let $e_{i}$ be an orthonormal basis for $T_{p}M$. If $e_{i}\wedge e_{j}$ diagonalize the curvature operator 
	$$\mathfrak{R}(e_{i}\wedge e_{j})=\la_{ij}e_{i}\wedge e_{j}$$
	then for any 2-plane $\pi \subset T_pM$, the sectional curvature satisfies
	$$sec(\pi)\in[\min{\la_{ij}},\max\la_{ij}].$$
\end{proposition}
\begin{proof}
	Given any orthonormal basis ${U,V}$ for $\pi$, we have 
	$$\sec(\pi)=g(\mathfrak{R}(U\wedge V),(U\wedge V)),$$
The result follows immediately from the spectral decomposition of $\mathfrak{R}$. 
\end{proof}
\begin{theorem}\label{T4}
	Let $(M,g)$ be a compact locally conformally flat manifold of dimension $n$, $n\geq4$. 
	\begin{itemize}
	\item[(1)]
	If the metric $g$ has positive sectional curvature, then the Betti numbers satisfy
 $$b_{p}(X)=0,\ \forall1\leq p\leq n-1.$$
In particular, if $n$ is even, then $\chi(M)=2$.
\item[(2)] If the metric $g$ has negative (resp. nonpositive) sectional curvature and $n$ is even, then Euler characteristic  of $M$ obeys
$$(-1)^{\frac{n}{2}}\chi(M)>0 (resp.\geq0).$$
\end{itemize} 
\end{theorem}
\begin{proof}

	\begin{itemize}
		\item[(1)]Since the sectional curvature of $g$ is positive, we know the curvature operator $\mathfrak{R}$ is also positive. By a result of Meyer \cite{Mey71}, the Betti numbers of $M$ are rational (co)homology spheres.
		\item[(2)]If the sectional curvature of $g$ is negative (resp. nonpositive), then the curvature operator $\mathfrak{R}$ is negative (resp. nonpositive).  By a well know result (cf.\cite{BK78,Kul72} or see Proposition \ref{P2}), the Euler characteristic  of $M$ obeys $(-1)^{\frac{n}{2}}\chi(M)>0$ (resp. $\geq0$).
	\end{itemize}

\end{proof}

\subsection{$k$th-intermediate Ricci curvature}
Let $(M,g)$ be an $n$-dimensional Riemannian  manifold. Suppose that the curvature operator $\mathfrak{R}$ having eigenvalues $\mu_{1}\leq\mu_2\leq\cdots\leq\mu_{N}$, where $N=\frac{n(n-1)}{2}$. We say that  $\mathfrak{R}$ is $k$-positive (resp. $k$-nonnegative) with $1\leq k\leq N-1$ (cf.\cite{NPW23,NPWW23,PW21a}) if
	$$\mu_{1}+\mu_{2}+\cdots\mu_{k}\geq0\ (\textrm{resp.} \geq 0).$$
Under a lower bound condition on the average of the lowest $(n-p)$ eigenvalues of the curvature operator, Petersen and Wink \cite{PW21a} established many general vanishing theorems and estimation theorems for the $p$-th Betti number.
\begin{definition}(cf.\cite{KM18,Mou22,RS05})\label{D2}
We say an $n$-dimensional Riemannian manifold $(M,g)$ has positive $k$th-intermediate
Ricci curvature for $k\in\{1,\cdots,n-1\}$ if, for any choice of orthonormal vectors $\{u,e_{1},\cdots,e_{n-1}\}$, the
sum of sectional curvatures
$$\sum_{i=1}^{k}\sec_{g}(u,e_{i})=\sum_{i=1}^{k}Rm(u,e_{i},u,e_{i})$$
is positive. We abbreviate this condition by writing $Ric_{k}(M,g)>0$, omitting $M$ or $g$ when they are understood.
\end{definition}
Note that $Ric_{1}>0$ is equivalent to positive sectional curvature, and $Ric_{n-1}>0$ is equivalent to positive Ricci curvature. Furthermore, if $Ric_{k}>0$, then $Ric_{l}>0$ for all $l
\geq k$. Thus $Ric_{2}>0$ is a strong condition on curvature in this hierarchy, second only to positive sectional curvature.

\begin{theorem}
Let $(M,g)$ be a compact locally conformally flat manifold of dimension $n$, $n\geq4$. If $M$ has positive $k$th-intermediate
Ricci curvature, then the curvature operator of $M$ is $k$-positive. Moreover
	\begin{itemize}
		\item [(1)]If $1\leq k\leq \lceil\frac{n}{2}\rceil$, then $b_{p}(M)=0$ for all $1\leq p\leq n-1$;
		\item [(2)] If $ \lceil\frac{n}{2}\rceil+1\leq k\leq n-1$, then 
		$b_{1}(M)=\cdots=b_{n-k}(M)$ for all $1\leq p\leq n-k$.
	\end{itemize}
\end{theorem}
\begin{proof}
Let $\{e_{i}\}_{i=1}^n$ be an orthonormal eigenbasis for the Ricci curvature. Following Corollary \ref{C5}, the sum of the smallest $k$ eigenvalues of the curvature operator satisfies
$$\sum_{i=1}^k\mu_i\geq\min\{\sum_{l=1}^{k}Rm(e_{i_{l}},e_{j_{l}},e_{i_{l}},e_{j_{l}} )\}\geq 0,$$
where the infimum is taken over all possible collections of $k$ distinct ordered pairs $(i_1,j_1), \dots, (i_k,j_k)$ with $i_l < j_l$ for each $1 \leq l \leq k$. Consequently, if a metric $g$ has positive $k$-th intermediate Ricci curvature, then its curvature operator is $k$-positive. The corresponding vanishing results for Betti numbers follow immediately from \cite[Theorem A]{PW21a}.
\end{proof}

\begin{corollary}
	Let $(M,g)$ be a compact locally conformally flat manifold manifold of dimension $n$, $n\geq3$. If the curvature condition
	 $$(\la_{1}+2\la_{2}+\la_{3})-\frac{2}{n-2}R_{g}$$ is quasi-positive, then 
	$M$ is diffeomorphic to a spherical space form. In particular, if $n=4$ or $6$, then $(M,g)$ is conformally equivalent to the sphere with its canonical metric.
\end{corollary}
\begin{proof}
	The curvature operator $\mathfrak{R}$ is $2$-quasi-positive, i.e. $(\la_{1}+2\la_{2}+\la_{3})-\frac{2}{n-2}R_{g}\geq 0$ for all $x\in M$ and there exists some $x_{0}\in M$ such that $(\la_{1}+2\la_{2}+\la_{3})-\frac{2}{n-2}R_{g}>0$. By using the method of B\"{o}hm-Wilking \cite{BW08}, Ni-Wu proved in \cite[Corollary 2.3]{NW07}, that along the Ricci flow, the curvature operator $\mathfrak{R}(g(t))$ of $g(t)$ is $2$-positive if $t>0$. In particular, by \cite{BW08} again, $M$ is diffeomorphic to a spherical
	space form.
	
Under our assumptions, the metric $g$ has non-negative scalar curvature. Moreover, when $n=4$ or $6$, by \cite[Theorem A]{Gur94}, the Euler characteristic satisfies $\chi(M)=2$ if and only if $(M,g)$ is conformally equivalent to the standard sphere $S^{n}$ endowed with its canonical metric.
\end{proof}

\section{Eigenvalue estimates and curvature operator positivity}
At any point $p\in M$, the decomposition we constructed for the curvature operator $\mathfrak{R}$ can be viewed as the sum of two Hermitian matrices. Therefore, we first introduce some properties of the eigenvalues of Hermitian matrices, as well as the definition of matrix norms.
\subsection{Eigenvalues of Hermitian matrix}
We first recall some definitions of the norms of Hermitian matrix (cf.\cite{Bha97,HJ12}). For any Hermitian matrix $A$, there are two fundamental norms:
\begin{itemize}
	\item Operator norm: $\|A\|=\sup_{|x|=1}|\langle x,Ax\rangle|$;
	\item Frobenius norm: $\|A\|_{2}=\left(\sum_{i,j=1}^{n}|a_{ij}|^{2}\right)^{\frac{1}{2}}=\left({\rm{tr}}(A^{\ast}A)\right)^{\frac{1}{2}}$.
\end{itemize}
The basic inequality between these two norms is:
$$\|A\|\leq\|A\|_{2}\leq\sqrt{n}\|A\|.$$
For eigenvalues $\la_{1}\leq\cdots\leq\la_{n}$ of $A$, we have the extremal characterizations:
\begin{equation*}
\begin{split}
&\la_{n}=\max\left\{\langle x,Ax\rangle:\|x\|=1 \right\},\\
&\la_{1}=\min\left\{\langle x,Ax\rangle:\|x\|=1 \right\}.\\
\end{split}
\end{equation*}
Moreover, the eigenvalues and the trace of $A$ satisfy the concentration inequality (cf.\cite[Page 24]{BG64})
\begin{equation*}
\left|\la_{i}-\frac{{\rm{tr}} A}{n}\right|\leq \left[\frac{n-1}{n}\left(\|A\|^{2}_{2}-\frac{|{\rm{tr}} A|^{2}}{n} \right) \right]^{\frac{1}{2}},
\end{equation*}
for all $i=1,\cdots,n$. Given any $1\leq k\leq n$, we also have
\begin{align}
\sum_{j=n-k+1}^{n}\la_{j}&=\max\sum_{j=1}^{k}\langle x_{j},Ax_{j}\rangle,\label{uppernk} \\
\sum_{j=1}^{k}\la_{j}&=\min\sum_{j=1}^{k}\langle x_j,Ax_j\rangle,\label{lowerk}
\end{align}
where the maximum and minimum are taken over all choices of orthonormal $k$-tuples $(x_{1},\cdots,x_{k})$. The equality \eqref{uppernk} is known as the Ky Fan's maximum principle. From \eqref{uppernk} and \eqref{lowerk}, we can conclude the following relations between eigenvalues of Hermitian matrices $A,B$ and $A+B$.
\begin{lemma}\label{L1}(cf.\cite[Page 35]{Bha97})
Let $A,B$ be two Hermitian matrices with eigenvalues $\la_{1}(A)\leq\cdots\leq\la_{n}(A)$ and $\la_{1}(B)\leq\cdots\leq\la_{n}(B)$, respectively. Then
\begin{itemize}
	\item[(a)]Subadditivity of lower sum:
$$\sum_{j=1}^{k}\la_{j}(A+B)\geq\sum_{j=1}^{k}(\la_{j}(A)+\la_{j}(B)),$$
\item[(b)]Superadditivity of upper sums:
$$\sum_{j=n-k+1}^{n}\la_{j}(A+B)\leq\sum_{j=n-k+1}^{n}(\la_{j}(A)+\la_{j}(B)).$$
\end{itemize}
\end{lemma}

\subsection{Weyl curvature and positively of curvature operator}
The curvature operator $\mathfrak{R}$ admits a  decomposition into sectional and Weyl components:
$$\overline{\mathfrak{S}}=\left(\begin{array}{ccc}R_{1212} & \cdots & 0 \\ \vdots & \ddots & \vdots \\0 & \cdots & R_{(n-1)n(n-1)n}\end{array}\right)$$
and
$$\overline{\mathfrak{W}}=\left(\begin{array}{ccc}0 &\cdots  & \mathcal{W}_{ijkl} \\ \vdots &\ddots & \vdots \\ \mathcal{W}_{ijkl}  & \cdots & 0\end{array}\right),$$
where $\overline{\mathfrak{S}}$ contains the sectional curvature components and $\overline{\mathfrak{W}}$ represents the trace-free Weyl curvature.

Let $\mu_1 \leq \cdots \leq \mu_N$ (respectively $\bar{\nu}_1 \leq \cdots \leq \bar{\nu}_N$), where $N = \frac{n(n-1)}{2}$, denote the eigenvalues of $\mathfrak{R}$ (respectively $\overline{\mathfrak{W}}$).

To estimate the sum of the eigenvalues of curvature operator, we first establish the following inequality for real numbers.
\begin{lemma}\label{L4}
	Let $a_{1} \leq \cdots \leq a_{N}$ be a sequence of real numbers with $\sum_{i=1}^{N} a_{i} = 0$. Then for any $1 \leq k \leq N$, 
	$$\sum_{i=1}^{k}a_{i}\geq-\sqrt{\frac{k(N-k)}{N}}(\sum_{i=1}^{N}a^{2}_{i})^{\frac{1}{2}}.$$
	Equality holds if and only if either:
	\begin{itemize}
		\item[(1)] $a_{1} = \cdots = a_{k} = -(N-k)$, and $a_{k+1} = \cdots = a_{N} = k$, or
		\item [(2)]$a_{1} = \cdots = a_{N} = 0$.
	\end{itemize}
\end{lemma}
\begin{proof}
	Applying the Cauchy-Schwarz inequality to the first $k$ and last $N-k$ terms respectively yields:
	\begin{equation*}
	\begin{split}
	&|\sum_{i=1}^{k}a_{i}|\leq \sqrt{k}(\sum_{i=1}^{k}a^{2}_{i})^{\frac{1}{2}},\\
	&|\sum_{i=k+1}^{N}a_{i}|\leq \sqrt{N-k}(\sum_{i=k+1}^{N}a^{2}_{i})^{\frac{1}{2}}.\\
	\end{split}
	\end{equation*}
	From the zero-sum condition $\sum_{i=1}^{N} a_{i} = 0$, we have $\sum_{i=k+1}^{N} a_{i} = -\sum_{i=1}^{k} a_{i}$. Combining this with the above inequalities gives:
	\begin{equation*}
	\frac{1}{k}(\sum_{i=1}^{k}a_{i})^{2}+\frac{1}{N-k}(\sum_{i=1}^{k}a_{i})^{2}\leq\sum_{i=1}^{N}a^{2}_{i}.
	\end{equation*}
	Since $\sum_{i=1}^{k} a_{i} \leq 0$, we obtain the desired inequality: 
	$$\sum_{i=1}^{k}a_{i}\geq-\sqrt{\frac{k(N-k)}{N}}(\sum_{i=1}^{N}a^{2}_{i})^{\frac{1}{2}}.$$
	The equality conditions follow from examining when both Cauchy-Schwarz inequalities become equalities, which requires the claimed proportional relationships between the $a_i$.
\end{proof}
We have the following eigenvalue estimates for curvature operator. The Lemma \ref{L4} yields the
\begin{proposition}\label{P4}
Let $(M,g)$ be a compact manifold of dimension $n$, $n\geq4$. Then  
\begin{itemize}
	\item[(1)]If the $k$th-intermediate Ricci curvature satisfies $Ric_{k}\geq ka$, then 
	$$\sum_{i=1}^{k}\mu_{i}\geq ka-\sqrt{\frac{k(N-k)}{N}}|\mathcal{W}_{g}|,$$
	\item[(2)]If the $k$th-intermediate Ricci curvature satisfies $Ric_{k}\leq ka$,
	$$\sum_{j=N-k+1}^{N}\mu_{j}\leq ka+\sqrt{\frac{k(N-k)}{N}}|\mathcal{W}_{g}|.$$
\end{itemize}
\end{proposition}
\begin{proof}
Using Lemma \ref{L4} with the trace-free property $tr \overline{\mathfrak{W}} = 0$, we obtain:
	$$-\sum_{i=1}^{k}\bar{\nu}_{i}\leq \sqrt{\frac{k(N-k)}{N}}\|\overline{\mathfrak{W}}\|_{2}\leq \sqrt{\frac{k(N-k)}{N}}|\mathcal{W}_{g}|.$$
Combining with Lemma \ref{L1} yields the lower bound
	$$\sum_{i=1}^{k}\mu_{i}\geq ka-\sqrt{\frac{k(N-k)}{N}}|\mathcal{W}_{g}|.$$
For the upper bound, consider the reversed eigenvalue sequence $-\bar{\nu}_N \leq \cdots \leq -\bar{\nu}_1$. By the Lemma \ref{L4}, we obtain
$$\sum_{i=N-k+1}^{N}(-\bar{\nu}_{i})\geq -\sqrt{\frac{k(N-k)}{N}}|\mathcal{W}_{g}|,$$
i.e.,$$\sum_{i=N-k+1}^{N}\bar{\nu}_{i}\leq \sqrt{\frac{k(N-k)}{N}}|\mathcal{W}_{g}|.$$
Following the same approach, we can derive the upper bound estimates for the largest $k$ eigenvalues.
\end{proof}
\begin{corollary}\label{C1}
	Let $(M,g)$ be a compact manifold of dimension $n$, $n\geq4$, with positive (resp. negative) sectional curvature $sec_{g}\geq a>0$ (resp. $sec_{g}\leq -a<0$). If the Weyl curvature $\mathcal{W}_{g}$ satisfies
	$$|\mathcal{W}_{g}|<\sqrt{\frac{kN}{N-k}}a,$$
	where $N=\frac{n(n-1)}{2}$,	then the curvature operator is $k$-positive (resp. $k$-negative). Furthermore, 	
	\begin{itemize}
		\item[(1)]If $1\leq k\leq \lceil\frac{n}{2}\rceil$, then $b_{p}(M)=0$ for all $1\leq p\leq n-1$;
		\item[(2)]If $ \lceil\frac{n}{2}\rceil+1\leq k\leq n-1$, then 
		$b_{p}(M)=0$ for all $1\leq p\leq n-k$.
	\end{itemize}
	\end{corollary}
\begin{proof}
	Following Proposition \ref{P4}, the curvature condition implies the $k$-positivity of the curvature operator. The topological conclusions follow from \cite[Theorem A]{PW21a}.
\end{proof}

\begin{proof}[\textbf{Proof of Theorem \ref{T2}}]
Under the curvature condition (where $k=1$), Corollary \ref{C1} implies that the curvature operator is positive (resp. negative). Consequently, Proposition \ref{P2} yields 
$${\rm{sign}}\chi(M)=(-1)^{n}.$$
\end{proof}

\subsection{Estimates of Betti numbers }
Building on fundamental work by Li \cite{Li80} and Gallot \cite{Gal81}, we establish bounds for the Betti numbers of compact Riemannian manifolds with constraints on the $k$th-intermediate Ricci curvature, Weyl curvature, and diameter.
\begin{theorem}
Let $(M,g)$ be a compact Riemannian manifold of dimension $n$, $n\geq4$.  Suppose that 
	$$Ric_{k}-\sqrt{\frac{k(N-k)}{N}}|\mathcal{W}_{g}|\geq ka\quad and\quad {\rm{diam}}(M)\leq D,$$	
	where $\lceil\frac{n}{2} \rceil\leq k\leq n-1$. Then there exists a constant $C(n,aD^{2})>0$ such that for all $1\leq p\leq n-k$,
	$$b_{p}(M)\leq {\binom{n}{p}}\exp\bigg{(}C(n,aD^{2})\cdot\sqrt{-aD^{2}p(n-p)}\bigg{)}.$$
	In particular, when $aD^{2}\geq-\varepsilon(n)$ of some $\varepsilon(n)$, we have $b_{p}(M)\leq {\binom{n}{p}}$. 
\end{theorem}
\begin{proof}
	The curvature assumption implies for the curvature operator eigenvalues  $\mu_{1}\leq\cdots \mu_{N}$, $N=\frac{n(n-1)}{2}$, of curvature operator $\mathfrak{R}$ (see Corollary \ref{C5}): 
	$$\frac{1}{n-p}(\mu_{1}+\cdots+\mu_{n-p})\geq\frac{1}{k}(\min Ric_{k}-\sqrt{\frac{k(N-k)}{N}}|\mathcal{W}_{g}|)\geq a.$$
	By \cite[Lemma 2.1]{PW21a},  any $p$-form $u\in\Om^{p}(X)$, $1\leq p\leq n-k$, satisfies:
	$$g(\mathfrak{Ric}(u),u)\geq p(n-p)a|u|^2.$$
	Following \cite[Theorem 1.9]{PW21a}, we obtain that the Betti number $b_{p}(M)$ is bounded by $${\binom{n}{p}}\exp\bigg{(}C(n,aD^{2})\cdot\sqrt{-aD^{2}p(n-p)}\bigg{)}.$$
\end{proof}
Building on the work of Huang and Tan \cite{HT25}, who established that a compact $2n$-manifold with uniformly bounded negative average of the lowest $n$ curvature operator eigenvalues and $b_1(M) \geq 1$ satisfies $\chi(M)=0$, we obtain the following refined result:
\begin{theorem}
	Let $(M,g)$ be a compact manifold of even dimension $2n$, $n\geq2$, with the first Betti number $b_{1}(M)\geq1$. Then there exists a positive constant $\varepsilon(n)$ such that if the Riemannian curvature tensor satisfies
	$$\bigg{(} Ric_{n}-\sqrt{\frac{n(n-3)}{n-1}  }|\mathcal{W}_{g}|\bigg{)}{\rm{diam}}^{2}(M)\geq-\varepsilon(n),$$
	then the Euler characteristic of $M$ obeys $\chi(M)=0$.    
\end{theorem} 
\begin{proof}
The curvature hypothesis implies the following estimate for the eigenvalues  $\mu_1\leq\cdots\leq\mu_{N}$ of the curvature operator $\mathfrak{R}$:
	$$(\mu_{1}+\cdots+\mu_{n}){\rm{diam}}^{2}(M)\geq\bigg{(}\min Ric_{n}-\sqrt{\frac{n(n-3)}{n-1}  }|\mathcal{W}_{g}|\bigg{)}{\rm{diam}}^{2}(M)\geq-\varepsilon(n).$$
The conclusion follows immediately from \cite[Theorem 1.1]{HT25}, which guarantees the strict inequality $(-1)^n\chi(M) > 0$ under these conditions.
	
\end{proof}

\section{ Integral curvature bounds and topological rigidity }
\subsection{Integral curvature bounds and Betti numbers }
Relaxing the pointwise boundedness of the Weyl curvature to an integral bound does not guarantee the preservation of curvature operator positivity. In the four-dimensional case, the Gauss-Bonnet-Chern formula establishes a fundamental relationship between the $L^{2}$-norm of the Weyl curvature and the manifold's Euler characteristic.
\begin{proposition}
Let $(M,g)$ be a compact Riemannian $4$-manifold with positive Ricci curvature. Suppose that the Weyl curvature obeys
$$\int_{M}|\mathcal{W}_{g}|^{2}<8\pi^{2},$$
then the Euler characteristic of $M$ obeys $\chi(M)=2$.
\end{proposition}
\begin{proof}
The curvature condition immediately yields $b_{1}(M)=0$. Consequently, the Euler characteristic simplifies to
$$\chi(M)=2+b_{2}(M).$$
Applying the topological constraint from \cite[Theorem 1.1]{Ses06} or \cite[Lemma 2.5]{Gur94}, we obtain 
$$\frac{1}{8\pi^{2}}\int_{M}|\mathcal{W}_{g}|^{2}\geq\chi(M)-2.$$
Combining these results gives $2\leq \chi(M)<3$, which forces $\chi(M)=2$ by integrality.
\end{proof}
For higher dimensions, we employ Sobolev inequalities to establish vanishing theorems under integral curvature conditions. Our proof proceeds as follows:
\begin{proof}[\textbf{Proof of Theorem \ref{T6}}]
The lower bound on Ricci curvature follows directly from our hypotheses:
$$Ric_{g}\geq(n-1)\frac{\min Ric_{k} }{k}\geq (n-1)a$$
For any smooth function $f$ on an compact $n$-manifold with Ricci curvature satisfying $Ric_{g}\geq(n-1)a$, Ilias' Sobolev inequality \cite[Page 546]{HV97} gives:
\begin{equation}\label{E4}
\|f\|^{2}_{L^{\frac{2n}{n-2}}(M)}\leq\frac{1}{[{\rm{Vol}}(g)]^{\frac{2}{n} } }\bigg{(}\frac{4}{n(n-2)a}\|\na f\|^{2}_{L^{2}(M)}+\|f\|^{2}_{L^{2}(M)}\bigg{)}.
\end{equation}
The sum of smallest $k$, $1\leq k\leq n-1$, eigenvalues  of curvature operator $\mathfrak{R}$ has a lower bounded
	$$ka-\sqrt{\frac{k(N-k)}{N}}|\mathcal{W}_{g}|.$$
Consequently, by \cite[Lemma 2.1]{PW21a}, for any $p$-form $u$, $1\leq p\leq n-k$,
	\begin{equation*}
	g(\mathfrak{Ric}(u),u)\geq
	p(n-p)\left(a-\sqrt{\frac{1}{k}-\frac{1}{N} }|\mathcal{W}_{g}|\right).
	\end{equation*}
For any harmonic $p$-form $u$, applying the Weitzenböck formula and integrating over the manifold $M$, we obtain:
	\begin{equation*}
	\begin{split}
	0&=\|\na u\|^{2}_{L^{2}}+\int_{M}g\left(\mathfrak{Ric}(u),u\right)\\
	&\geq \|\na u\|^{2}_{L^{2}}+p(n-p)a\|u\|^{2}_{L^{2}}-p(n-p)\sqrt{\frac{1}{k}-\frac{1}{N} }\int_{M}|\mathcal{W}_{g}|\cdot|u|^{2}\\
	&\geq \|\na u\|^{2}_{L^{2}}+p(n-p)a\|u\|^{2}_{L^{2}}-p(n-p)\sqrt{\frac{1}{k}-\frac{1}{N} }\|\mathcal{W}_{g}\|_{L^{\frac{n}{2}}}\cdot\|u\|^{2}_{L^{\frac{2n}{n-2}}}\\
	&\geq \|\na u\|^{2}_{L^{2}}\left(1-\sqrt{\frac{N-k}{Nk} }\frac{4p(n-p)\|\mathcal{W}_{g}\|_{L^{\frac{n}{2}}}}{n(n-2)a[{\rm{Vol}}(g)]^{\frac{2}{n}}} \right)+p(n-p)\|u\|^{2}_{L^{2}}\left(a-\sqrt{\frac{N-k}{Nk} }\frac{\|\mathcal{W}_{g}\|_{L^{\frac{n}{2}}}}{[{\rm{Vol}}(g)]^{\frac{2}{n}}}\right).\\
	\end{split}
	\end{equation*}
	On the third line, we use the Holder inequality
	$$\int_{M}|f|\cdot |g|^{2}\leq \|f\|_{L^{\frac{n}{2} }}\|g\|^{2}_{L^{\frac{2n}{n-2} }}.$$
Now, assume the curvature bound 
	$$\bigg{(}\frac{\int_{M}|\mathcal{W}_{g}|^{\frac{n}{2}}}{{\rm{Vol}}(g)}\bigg{)}^{\frac{2}{n}}<\sqrt{\frac{Nk}{N-k}}\frac{n-2}{n}a.$$ 
This implies
 $$1-\sqrt{\frac{1}{k}-\frac{1}{N} }\frac{4p(n-p)\|\mathcal{W}_{g}\|_{L^{\frac{n}{2}}}}{n(n-2)a[{\rm{Vol}}(g)]^{\frac{2}{n}}}>1-\frac{4p(n-p)}{n(n-2)}\frac{n-2}{n}\geq0,$$
 for all $1\leq p\leq n-1$.
Therefore $u=0$, i.e, $b_{p}(M)=0$. This completes the proof of the vanishing theorem under integral curvature bounds.
\end{proof}

\subsection{Yamabe constant  and curvature operator }
Let $(M,g)$ be a compact, smooth Riemannian manifold of dimension $n\geq 3$. Consider the conformal metric $\tilde{g}=u^{\frac{4}{n-2}}g$ on $M$, where $u$ is a smooth positive function. If $R_{g}$ and $R_{\tilde{g}}$ denote the scalar curvatures of $(M,g)$ and $(M,\tilde{g})$, respectively, then we have
\begin{equation}
4\frac{n-1}{n-2}\De_{g}u+R_{g}u=R_{\tilde{g}}u^{\frac{n+2}{n-2}}.
\end{equation}
Thus $\tilde{g}$ has constant scalar curvature $c$ if and only if $u$ satisfies the well-known Yamabe equation
\begin{equation}\label{E2}
4\frac{n-1}{n-2}\De_{g}u+R_{g}u=cu^{\frac{n+2}{n-2}}.
\end{equation}
For every positive function $u\in C^{\infty}$,
\begin{equation*}
Q_{g}(u):=\frac{\int_{M}R_{\tilde{ g}}d{\rm{Vol}}_{ \tilde{g}}      }   { (\int_{M}d{\rm{Vol}}_{ \tilde{g}})^{\frac{n-2}{n} }      }= \frac{\int_{M}(4\frac{n-1}{n-2}|\na u|^{2}+R_{g}u^{2})d{\rm{Vol}}_{g}      }   { (\int_{M} |u|^{\frac{2n}{n-2}}d{\rm{Vol}}_{g})^{\frac{n-2}{n} } }
\end{equation*}
is called the Yamabe functional. A positive function $u$ is a critical point of $Q_{g}$ if and only if $u$ satisfies (\ref{E2}) with $$c=\frac{Q_{g}(u)} { (\int_{M} |u|^{\frac{2n}{n-2}}d{\rm{Vol}}_{g})^{\frac{2}{n} } }.$$ The so-called Yamabe constant
$$\la(g):=\inf\{Q_{g}:u\in C^{\infty}, u>0 \}$$ 
is a conformal invariant. 

We can observe that when the Ricci curvature has a positive lower bound, the Yamabe constant also admits a positive lower bound. The proof is based on Ilias's construction of the Sobolev inequality under the condition of a positive lower bound for the Ricci curvature (cf.\cite{Ili83}).
\begin{proposition}\label{P6}
	Let $(M,g)$ be a closed, connected Riemannian manifold of dimension $n\geq 3$. If the Ricci curvature obeys $Ric_{g}\geq(n-1)a$, then 	Yamabe constant has a lower bounded, that is 
	$$\la(g)\geq n(n-1)a[{\rm{Vol}}(g)]^{\frac{2}{n}}.$$
\end{proposition}
\begin{proof}
The Ricci lower bound implies the scalar curvature $R_{g}\geq n(n-1)a$. Combining this with Ilias' Sobolev inequality (\ref{E4}) yields:
	\begin{equation*}
	\begin{split}
	Q_{g}(u)&\geq \frac{\int_{M}(4\frac{n-1}{n-2}|\na u|^{2}+n(n-1)au^{2})d{\rm{Vol}}_{g}      }   { (\int_{M} |u|^{\frac{2n}{n-2}}d{\rm{Vol}}_{g})^{\frac{n-2}{n} } }\\
	&\geq \frac{\int_{M}(4\frac{n-1}{n-2}|\na u|^{2}+n(n-1)au^{2})d{\rm{Vol}}_{g}      }   { \frac{1}{[Vol(g)]^{\frac{2}{n} } }\bigg{(}\frac{4}{n(n-2)a}\|\na f\|^{2}_{L^{2}(X)}+\|f\|^{2}_{L^{2}(X)}\bigg{)}}\\
	&\geq n(n-1)a[{\rm{Vol}}(g)]^{\frac{2}{n}}. \\
	\end{split}
	\end{equation*}
	Therefore $\la(g):=\inf\{ Q_{g}u:u\in C^{\infty},u>0\}\geq n(n-1)a[{\rm{Vol}}(g)]^{\frac{2}{n}}$.
\end{proof}

On any Riemannian manifold $(M,g)$, the curvature tensor admits the orthogonal decomposition:
$$Rm=\frac{R_{g}}{2n(n-1)}g\odot g+\frac{1}{n-2}\overset{\circ}{Ric}_{g}\odot g+\mathcal{W}_{g}$$
where $\overset{\circ}{Ric}_{g}=Ric_{g}-\frac{R_{g}}{n}\cdot g$ is the trace-free Ricci tensor and $\mathcal{W}_{g}$ denotes the Weyl
part. 	We also call $Z_{g}=\frac{1}{n-2}\overset{\circ}{Ric}_{g}\odot g+\mathcal{W}_{g}$ the concircular curvature tensor (cf.\cite{HV97}), and
$$|Z_{g}|^{2}=\frac{4}{n-2}|\overset{\circ}{Ric}_{g}|^{2}+|\mathcal{W}_{g}|^{2}.$$
\begin{lemma}
Denote by $\overset{\circ}{\mathfrak{Ric}}$ the curvature operator induced by $\frac{1}{n-2}\overset{\circ}{Ric}_g \odot g$. Then
$$\|\overset{\circ}{\mathfrak{Ric}}\|^{2}_{2}=\frac{1}{n-2}|\overset{\circ}{Ric}_{g}|^{2}.$$
\end{lemma}
\begin{proof}
Let $\{e_{i}\}_{1,\cdots,n}$ be an orthonormal eigenbasis for the Ricci tensor $Ric_g$ with corresponding eigenvalues $\{\la_{i}\}_{1,\cdots,n}$. By Lemma \ref{P5}, the operator $\overset{\circ}{\mathfrak{Ric}}$ acts on bivectors as
$$\overset{\circ}{\mathfrak{Ric}}(e_{i}\wedge e_{j})=\frac{1}{n-2}(\la_{i}+\la_{j}-\frac{2R_{g}}{n})(e_{i}\wedge e_{j}).$$
Thus, the eigenvalues of $\overset{\circ}{\mathfrak{Ric}}$ are $\frac{1}{n-2}(\la_{i}+\la_{j}-\frac{2R_{g}}{n})$ for $1 \leq i < j \leq n$.	
	
We now compute the Frobenius norm
\begin{equation*}
\begin{split}
\|\overset{\circ}{\mathfrak{Ric}}\|_{2}^{2}&=\frac{1}{(n-2)^{2} }\sum_{i<j}\left(\la_{i}+\la_{j}-\frac{2R_{g}}{n}\right)^{2}\\
&=\frac{1}{(n-2)^{2}}\sum_{i<j}\left[\left(\la_{i}-\frac{R_{g}}{n}\right)^{2}+\left(\la_{i}-\frac{R_{g}}{n}\right)^{2}\right]+\frac{2}{(n-2)^{2} }\sum_{i<j}\left(\la_{i}-\frac{R_{g}}{n}\right)\left(\la_{j}-\frac{R_{g}}{n}\right)\\
&=\frac{1}{(n-2)^{2}}\cdot(n-1)\sum_{i=1}^n\left(\la_{i}-\frac{R_{g}}{n} \right)^{2}+\frac{1}{(n-2)^{2}}\left[\left(\sum_{i=1}^{n}(\la_{i}-\frac{R_{g}}{n})\right)^2-\sum_{i=1}^{n}\left(\la_{i}-\frac{R_{g}}{n}\right)^2\right]\\
&=\frac{1}{n-2}\sum_{i=1}^{n}\left(\la_{i}-\frac{R_{g}}{n} \right)^{2}\\
&=\frac{1}{n-2}|\overset{\circ}{Ric}_{g}|^{2}.
\end{split}
\end{equation*}
The fourth equality follows from the trace-free property of  $\overset{\circ}{\mathfrak{Ric}}$:
$$tr\overset{\circ}{\mathfrak{Ric}}=\frac{1}{n-2}\sum_{i<j}(\la_{i}+\la_{j}-\frac{2R_{g}}{n})=\frac{n-1}{n-2}\sum_{i=1}^{n}(\la_{i}-\frac{R_{g}}{n})=0.$$	
\end{proof}

\begin{proposition}\label{P8}
		Let $(M,g)$ be a compact Riemannian  manifold of dimension $n$, $n\geq4$. The sum of the smallest $k$ eigenvalues of the curvature operator satisfies:
			$$\sum_{i=1}^{k}\mu_{i}\geq \frac{kR_{g}}{n(n-1)}-\sqrt{\frac{k(N-k)}{N}}(\sqrt{\frac{ 1}{n-2}}|\overset{\circ}{Ric}_{g}|+|\mathcal{W}_{g}|),$$
			where $N=\frac{n(n-1)}{2}$ is the dimension of $\Lambda^2TM$.
		\end{proposition}
	\begin{proof}
Let $\{e_{i}\}_{1,\cdots,n}$ be an orthonormal eigenbasis for the Ricci curvature tensor. Then the curvature operator $$\mathfrak{R} = \mathfrak{S}+\mathfrak{W},$$
where the Schouten tenor party further decomposes as:
$$\mathfrak{S}=\frac{R_{g}}{n(n-1)}Id+\overset{\circ}{\mathfrak{Ric}}.$$
The eigenvalue estimates follow from:
$$\sum_{i=1}^{k}(\mu_{i}-\frac{R_{g}}{n(n-1)})\geq\sum_{i=1}^{k}\overset{\circ}{\la}_{i}+\sum_{i=1}^{k}\nu_{i},$$
where $\overset{\circ}{\la}_{1}\leq\cdots\leq\overset{\circ}{\la}_{N}$ denote the eigenvalues of $\overset{\circ}{\mathfrak{Ric}}$.

Applying Lemma \ref{L4} and using the trace-free conditions $tr(\overset{\circ}{\mathfrak{Ric}}) = tr(\mathfrak{W}) = 0$, we obtain the desired eigenvalue estimate:
\begin{equation*}
\begin{split}
\sum_{i=1}^{k}\mu_{i}&\geq\frac{kR_{g}}{n(n-1)}-\sqrt{\frac{k(N-k)}{N}}(\|\overset{\circ}{\mathfrak{Ric}}\|_{2}+\|\mathfrak{W}\|_{2})\\
&\geq \frac{kR_{g}}{n(n-1)}-\sqrt{\frac{k(N-k)}{N}}(\sqrt{\frac{ 1}{n-2}}|\overset{\circ}{Ric}_{g}|+|\mathcal{W}_{g}|).
\end{split}
\end{equation*}
\end{proof}	
\begin{lemma}\label{L5}
For any integer $n \geq 8$ and integer $p$ satisfying $1 \leq p \leq \frac{n}{2}-2$, the following inequality holds:
$$1+\frac{1}{n-p}>\frac{4p(n-p)}{n(n-2)}.$$
\end{lemma}
\begin{proof}
We proceed by considering two cases based on the value of $p$:\\
\textbf{Case 1:} $1 \leq p \leq \frac{n}{2} - \sqrt{\frac{n}{2}}$

In this range, we have $\frac{4p(n - p)}{n(n - 2)} \leq 1$ (this follows from straightforward algebraic manipulation). Since $\frac{1}{n - p} > 0$, the inequality holds strictly.\\
\textbf{Case 2:} $\frac{n}{2} - \sqrt{\frac{n}{2}} < p \leq \frac{n}{2} - 2$

Let $p = \frac{n}{2} - k$, where $k$ satisfies $2 \leq k < \sqrt{\frac{n}{2}}$. Substituting this into the original inequality, we obtain:
\begin{equation*}
\begin{split}
1+\frac{1}{n-p}-\frac{4p(n-p) }{n(n-2)}&=1+\frac{2}{n+2k}-\frac{n^{2}-4k^{2} }{n(n-2)}\\
&=\frac{2(4k^{3}+2k^{2}n-2kn-2n)}{(n+2k)n(n-2)}.
\end{split}
\end{equation*}
Define $F(k) = 4k^3 + 2k^2n - 2kn - 2n$. The derivative with respect to $k$ is:
$$F'(k)=12k^{2}+4nk-2n>0,\ {\rm{for}\ all}\ k\geq1.$$
Thus, $F(k)$ is strictly increasing for $k \geq 1$. Evaluating at the endpoints:
$$F(1) = 4 - 4n < 0\  and\ F(2) = 32 + 2n> 0.$$
For all $k \geq 2$, the positivity of $F(k)$ follows from the monotonicity. Therefore, the numerator is positive in our range of interest, which completes the proof.	
\end{proof}
For compact Riemannian manifold $(M,g)$ with positive Yamabe constant $\la(g)$, we have the following conformal  Sobolev inequality
\begin{equation}\label{confSobo}
\la(g)\|u\|^{2}_{L^{\frac{2n}{n-2} } }\leq\frac{4(n-1)}{n-2}\|\na u\|^{2}_{L^{2}}+\int_{M}R_{g}|u|^{2}.
\end{equation}
This inequality will play an crucial role in the following proof of Theorems \ref{T11} and \ref{T1}.
\begin{proof}[\textbf{Proof of Theorem \ref{T11}}]
By \cite[Lemma 2.1]{PW21a}, for any $p$-form $u$ with $1\leq p\leq n-k$, the following  estimate holds
	\begin{equation*}
	\begin{split}
	g(\mathfrak{Ric}(u),u)&\geq p(n-p)(\frac{1}{k}\sum_{i=1}^{k}\mu_{i})|u|^{2}\\
	&\geq p(n-p)\bigg{(}\frac{R_{g}}{n(n-1)}-\sqrt{\frac{N-k}{Nk}}(\sqrt{\frac{ 1}{n-2}}|\overset{\circ}{Ric}_{g}|+|\mathcal{W}_{g}|)\bigg{)}|u|^{2}\\
	&= p(n-p)\bigg{(}\frac{R_{g}}{n(n-1)}-C(N,k)\mathcal{Z}_{g}\bigg{)}|u|^{2},\\
	\end{split}
	\end{equation*}
	where we set
	$$C(N,k)= \sqrt{\frac{N-k}{Nk}},\quad and\quad 
	\mathcal{Z}_g= \frac{|\mathring{Ric}_g|}{\sqrt{n-2}} + |\mathcal{W}_g|.$$
Using the H\"older inequality and the conformal  Sobolev inequality \eqref{confSobo} for $|u|$, we can get
	\begin{align*}
\int_{M}g(\mathfrak{Ric}(u),u)
	\geq& p(n-p)\int_{M}\left(\frac{R_{g}}{n(n-1)}-C(N,k)\mathcal{Z}_{g}\right)|u|^{2}\\
	\geq& p(n-p)\left[\left(\int_{M}\frac{R_{g}|u|^{2}}{n(n-1)}\right)-C(N,k)\|\mathcal{Z}_{g}\|_{L^{ \frac{n}{2} }}\|u\|^{2}_{L^{\frac{2n}{n-2} }} \right]\\
	\geq &p(n-p)\left[\left(\int_{M}\frac{R_{g}|u|^{2}}{n(n-1)}\right)-\frac{C(N,k)}{\la(g)}\|\mathcal{Z}_{g}\|_{L^{ \frac{n}{2} }}  \left(\frac{4(n-1)}{n-2}\|\na| u|\|^{2}_{L^{2}}+\int_{M}R_{g}|u|^{2}\right)\right]\\
	\geq&-p(n-p)\frac{C(N,k)}{\la(g)}\|\mathcal{Z}_{g}\|_{L^{ \frac{n}{2} }}\frac{4(n-1)}{n-2} \|\na |u|\|^{2}_{L^{2}}\\
	&+\int_{M}R_{g}|u|^{2}\left(\frac{p(n-p)}{n(n-1)}-p(n-p)\frac{C(N,k)}{\la(g)}\|\mathcal{Z}_{g}\|_{L^{ \frac{n}{2} }}\right).
\end{align*}
	
Now, let $u$ be a harmonic harmonic $p$-form, where $1\leq p\leq \frac{n}{2}$. From the refined Kato's inequality (cf. \cite{CGH00}), we have:
$$|\na u|^{2}\geq \frac{n-p+1}{n-p}|\na|u||^{2}.$$
For any $p\leq\min \{n-k,\frac{n}{2}\}$, substituting the Kato's inequality, we derive:
	\begin{equation*}
	\begin{split}
	0=&\|\na u\|^{2}_{L^{2}}+\int_{M}g(\mathfrak{Ric}(u),u)\\
\geq&\left(1+\frac{1}{n-p}\right) \|\na |u|\|^{2}_{L^{2}}-p(n-p)\frac{C(N,k)}{\la(g)}\|\mathcal{Z}_{g}\|_{L^{ \frac{n}{2} }}\frac{4(n-1)}{n-2}\|\na |u|\|^{2}_{L^{2}}\\
	&+\int_{M}R_{g}|u|^{2}\left(\frac{p(n-p)}{n(n-1)}-p(n-p)\frac{C(N,k)}{\la(g)}\|\mathcal{Z}_{g}\|_{L^{ \frac{n}{2} }}\right).\\
\geq&\bigg{(}\frac{p(n-p)}{n(n-1)}-p(n-p)\frac{C(N,k)}{\la(g)}\|\mathcal{Z}_{g}\|_{L^{ \frac{n}{2} }} \bigg{)}\bigg{(}\frac{4(n-1)}{n-2}\|\na |u|\|^{2}_{L^{2}}+\int_{M}R_{g}|u|^{2}\bigg{)}\\
&+\left(1+\frac{1}{n-p}-\frac{4(n-1)}{n-2}\frac{p(n-p)}{n(n-1)}\right)\|\na |u|\|^{2}_{L^{2}}\\
\geq&\bigg{(}\frac{p(n-p)}{n(n-1)}-p(n-p)\frac{C(N,k)}{\la(g)}\|\mathcal{Z}_{g}\|_{L^{ \frac{n}{2} }}\bigg{)}\la(g)\|u\|^{2}_{L^{\frac{2n}{n-2} } }+\left(1+\frac{1}{n-p}-\frac{4p(n-p)}{n(n-2)}\right)\|\na |u|\|^{2}_{L^{2}}.  
	\end{split}
	\end{equation*}
where the last inequality comes from the repeated use of the conformal  Sobolev inequality \eqref{confSobo} together with the  curvature condition 
	$$\|\mathcal{Z}_{g}\|_{L^{ \frac{n}{2} }}<\frac{\la(g)}{n(n-1)C(N,k)}.$$
	
By Lemma \ref{L5}, the coefficients of the terms $\|u\|^{2}_{L^{\frac{2n}{n-2} } }$ and $\|\na |u|\|^{2}_{L^{2}}$ are strictly positive for any $1\leq p\leq\min \{n-k,\frac{n}{2}-2\}$. Therefore, we conclude that $u=0$, i.e., $b_{p}(M)=0$. Hence,
\begin{itemize}
	\item[(1)] For $1\leq k\leq \lceil\frac{n}{2}\rceil$, the Betti numbers satisfy
	$b_{p}(M)=0 $ for all $1\leq p\leq \lceil\frac{n}{2}\rceil-2$.
 	\item[(2)] For $ \lceil\frac{n}{2}\rceil+2\leq k\leq n-1$, the Betti numbers satisfy
 	$b_{p}(M)=0 $ for all $1\leq p\leq\min\{n-k,\lceil\frac{n}{2}\rceil-2\}$.
\end{itemize}
\end{proof}
\begin{proof}[\textbf{Proof of Corollary \ref{C3}}]
For any harmonic $p$-form $u$ with $1\leq p\leq \min\{n-k,\frac{n}{2}\} $, the following inequality can be derived from the proof of Theorem \ref{T11}: 
	\begin{equation*}
	\begin{split}
	0=&\|\na u\|^{2}_{L^{2}}+\int_{M}g(\mathfrak{Ric}(u),u)\\	
	\geq &\|\na |u|\|^{2}_{L^{2}}\bigg{(}(1+\frac{1}{n-p}-p(n-p)C(N,k)\|\mathcal{Z}_{g}\|_{L^{ \frac{n}{2} }}\frac{1}{\la(g)}\frac{4(n-1)}{n-2}\bigg{)}\\
	&+\int_{M}R_{g}|u|^{2}\bigg{(}\frac{p(n-p)}{n(n-1)}-p(n-p)C(N,k)\|\mathcal{Z}_{g}\|_{L^{ \frac{n}{2} }}\frac{1}{\la(g)} \bigg{)}\\
	\end{split}
	\end{equation*}	
By $R_{g}>0$ and the given curvature conditions, the coefficients of both terms are strictly positive for all $p\leq \min\{n-k,\frac{n}{2}\}$. Therefore, we conclude that $u=0$, i.e., $b_{p}(M)=0$. 
\end{proof}
\subsection{Harmonic Weyl Curvature and integral gap}
For a closed Riemannian manifold with harmonic Weyl curvature, Tran \cite{Tra17} derived an integral gap result in dimensions higher than $5$. He used new Bochner-Weitzenb\"{o}ck-Lichnerowicz type formulas for the Weyl tensor, which are generalizations of identities in dimension $4$. Based on the  method of controlling the curvature term of the Lichnerowicz Laplacian on tensors as in \cite{PW21a,PW21b,PW22}, we also  show an integral gap result for harmonic Weyl curvature in this section.

	A Riemannian manifold $(M,g)$ is said to have harmonic Weyl curvature if its Weyl tensor satisfies the divergence-free condition:
	$$\delta\mathcal{W}_{g}=0.$$
	Since the divergence-free Weyl tensors satisfy the second Bianchi identity, Petersen and Wink \cite{PW22} established a Weitenb\"{o}ck formula for  harmonic Weyl curvature as follows.
\begin{proposition}(\cite[Proposition 2.2]{PW22})\label{P7}
	Let $(M,g)$ be a Riemannian manifold. If the Weyl curvature $\mathcal{W}_{g}$ is divergence free, then $\mathcal{W}_{g}$ satisfies the second Bianchi identity and
	$$\na^{\ast}\na\mathcal{W}_{g}+\frac{1}{2}\mathfrak{Ric}(\mathcal{W}_{g})=0.$$
\end{proposition}
\begin{lemma}\label{L3}
	Let $(M,g)$ be a Riemannian manifold. Then the Weyl curvature obeys
	$$g(\mathfrak{Ric}(\mathcal{W}_{g}),\mathcal{W}_{g})=g(\mathfrak{R}(\mathcal{W}^{\mathfrak{so}}_{g}),\mathcal{W}^{\mathfrak{so}}_{g}),$$
	where $\mathcal{W}^{\mathfrak{so}}_{g}$ is defined as in Definition \ref{D1}.
\end{lemma}
\begin{proof}
	This identity follows immediately from  \cite[Proposition 1.6]{PW21b}.
\end{proof}
Now we can prove the Theorem \ref{T1}. The proof of Case I is similar to Theorem \ref{T6}, and the proof of Case II is similar to Theorem \ref{T11}.
\begin{proof}[\textbf{Proof of Theorem \ref{T1}}]
	The key ingredients come from \cite[Lemma 2.2 and Proposition 2.5]{PW21a}, which yield the estimate:
	$$|L\mathcal{W}_{g}|^{2}\leq 8|\mathcal{W}_{g}|^{2}|L|^{2}=\frac{2}{n-1}|\mathcal{W}^{\mathfrak{so}}_{g}||L^{2}|$$
	for all $L\in\mathfrak{so}(TM)$ (or see \cite[Corollary 2.3]{PW22}). \\	
	\textbf{Case I}:	Under our curvature assumptions, we obtain a lower bound for the average of the smallest $\lfloor\frac{n-1}{2}\rfloor$ eigenvalues of the curvature operator
	$$\frac{1}{k}(ak-\sqrt{\frac{k(N-k)}{N}}|\mathcal{W}_{g}|).$$
	Applying Lemma \ref{L3} and \cite[Lemma 1.8]{PW21b}), we derive the inequality:
	\begin{equation*}
	\begin{split}
	g(\mathfrak{Ric}(\mathcal{W}_{g}),\mathcal{W}_{g})&=g(\mathfrak{R}(\mathcal{W}^{\mathfrak{so}}_{g}),\mathcal{W}^{\mathfrak{so}}_{g} )\\
	&\geq (a-\sqrt{\frac{1}{k}-\frac{1}{N} }|\mathcal{W}_{g}|)|\mathcal{W}^{\mathfrak{so}}_{g}|^{2}\\
	&=(a-\sqrt{\frac{1}{k}-\frac{1}{N} }|\mathcal{W}_{g}|)4(n-1)|\mathcal{W}_{g}|^{2}.
	\end{split}
	\end{equation*}
	Following Proposition \ref{P7}, we obtain the estimate:
	\begin{equation*}
	\begin{split}
	0&=\|\na\mathcal{W}_{g}\|^{2}_{L^{2}}+\frac{1}{2}\int_{M}g\left(\mathfrak{Ric}(\mathcal{W}_{g}),\mathcal{W}_{g}\right)\\
	&\geq \|\na \mathcal{W}_{g}\|^{2}_{L^{2}}+2(n-1)a\|\mathcal{W}_{g}\|^{2}_{L^{2}}-2(n-1)\sqrt{\frac{1}{k}-\frac{1}{N} }\int_{M}|\mathcal{W}_{g}|\cdot|\mathcal{W}_{g}|^{2}\\
	&\geq \|\na \mathcal{W}_{g}\|^{2}_{L^{2}}+2(n-1)a\|\mathcal{W}_{g}\|^{2}_{L^{2}}-2(n-1)\sqrt{\frac{1}{k}-\frac{1}{N} }\|\mathcal{W}_{g}\|_{L^{\frac{n}{2}}}\cdot\|\mathcal{W}_{g}\|^{2}_{L^{\frac{2n}{n-2}}}\\
	&\geq \|\na \mathcal{W}_{g}\|^{2}_{L^{2}}\left(1-\sqrt{\frac{N-k}{Nk} }\frac{8(n-1)\|\mathcal{W}_{g}\|_{L^{\frac{n}{2}}}}{n(n-2)a[Vol(g)]^{\frac{2}{n}}} \right)+2(n-1)\|\mathcal{W}_{g}\|^{2}_{L^{2}}\left(a-\sqrt{\frac{N-k}{Nk}}\frac{\|\mathcal{W}_{g}\|_{L^{\frac{n}{2}}}}{[Vol(g)]^{\frac{2}{n}}}\right).
	\end{split}
	\end{equation*}
	Consequently, if the Weyl curvature satisfies $$\bigg{(}\frac{\int_{M}|\mathcal{W}_{g}|^{\frac{n}{2}}}{Vol(g)}\bigg{)}^{\frac{2}{n}}<\min\{\sqrt{\frac{Nk}{N-k}}a,\frac{n(n-2)}{8(n-1)}\sqrt{\frac{Nk}{N-k}}a \},$$
	then we conclude that $\mathcal{W}_{g} \equiv 0$.\\
		\textbf{Case II}: By Proposition \ref{P8}, the average of the smallest $k := \lfloor \frac{n-1}{2} \rfloor$ eigenvalues of the curvature operator satisfies
		$$\frac{1}{k}\sum_{i=1}^{k}\mu_{i}\geq \frac{R_{g}}{n(n-1)}-\sqrt{\frac{k(N-k)}{N}}\left(\sqrt{\frac{ 1}{n-2}}|\overset{\circ}{Ric}_{g}|+|\mathcal{W}_{g}|\right).$$
		Applying Lemma \ref{L3} and \cite[Lemma 1.8]{PW21b}), we derive the inequality:
		\begin{equation*}
		\begin{split}
		g(\mathfrak{Ric}(\mathcal{W}_{g}),\mathcal{W}_{g})&=g(\mathfrak{R}(\mathcal{W}^{\mathfrak{so}}_{g}),\mathcal{W}^{\mathfrak{so}}_{g} )\\
		&\geq \left(\frac{R_{g}}{n(n-1)}-\sqrt{\frac{k(N-k)}{N}}\left(\sqrt{\frac{ 1}{n-2}}|\overset{\circ}{Ric}_{g}|+|\mathcal{W}_{g}|\right)\right)|\mathcal{W}^{\mathfrak{so}}_{g}|^{2}\\
		&=\bigg{(}\frac{R_{g}}{n(n-1)}-C(N,k)\mathcal{Z}_{g}\bigg{)}4(n-1)|\mathcal{W}_{g}|^{2}.
		\end{split}
		\end{equation*}
		Using the H\"older inequality and the conformal  Sobolev inequality \eqref{confSobo} for $|\mathcal{W}_{g}|$, we can get
		\begin{align*}
		\frac{1}{2}\int_{M}g(\mathfrak{Ric}(\mathcal{W}_{g}),\mathcal{W}_{g})=&2(n-1)\int_M\left(\frac{R_{g}}{n(n-1)}-C(N,k)\mathcal{Z}_{g}\right)|\mathcal{W}_{g}|^{2}.\\
		\geq& -2(n-1)C(N,k)\|\mathcal{Z}_{g}\|_{L^{ \frac{n}{2} }}\cdot\|\mathcal{W}_{g}\|^{2}_{L^{\frac{2n}{n-2}}}+\frac{2}{n}\int_{M}R_{g}|\mathcal{W}_{g}|^{2}\\
		\geq&-2(n-1)\frac{C(N,k)}{\la(g)}\|\mathcal{Z}_{g}\|_{L^{ \frac{n}{2} }}\left(\frac{4(n-1)}{n-2}\|\na|\mathcal{W}_{g}|\|^{2}_{L^{2}}+\int_{M}R_{g}|\mathcal{W}_{g}|^{2}\right)\\
		&+\frac{2}{n}\int_{M}R_{g}|\mathcal{W}_{g}|^{2}.
		\end{align*}
		For harmonic Weyl tensor, there is an improved Kato's inequality (cf.\cite{Bra00,Tra17})
		$$|\na\mathcal{W}_{g}|^{2}\geq\frac{n+1}{n-1}|\na|\mathcal{W}_{g}||^{2}.$$	
		Following Proposition \ref{P7}, we obtain the estimate:
		\begin{equation*}
		\begin{split}
		0=&\|\na\mathcal{W}_{g}\|^{2}_{L^{2}}+\frac{1}{2}\int_{M}g(\mathfrak{Ric}(\mathcal{W}_{g}),\mathcal{W}_{g})\\
		\geq&\frac{n+1}{n-1}\|\na |\mathcal{W}_{g}|\|^{2}_{L^{2}}+\frac{2}{n}\int_{M}R_{g}|\mathcal{W}_{g}|^{2}\\
		&-2(n-1)\frac{C(N,k)}{\la(g)}\|\mathcal{Z}_{g}\|_{L^{ \frac{n}{2} }}\left(\frac{4(n-1)}{n-2 }\|\na|\mathcal{W}_{g}|\|^{2}_{L^{2}}+\int_{M}R_{g}|\mathcal{W}_{g}|^{2}\right)\\
		\geq &\bigg{(}\frac{2}{n}-\frac{2(n-1)}{\la(g)}C(N,k)\|\mathcal{Z}_{g}\|_{L^{ \frac{n}{2} }}\bigg{)}\bigg{(}\frac{4(n-1)}{n-2}\|\na|\mathcal{W}_{g}|\|^{2}_{L^{2}}+\int_{M}R_{g}|\mathcal{W}_{g}|^{2}\bigg{)}\\
		&+\bigg{(}\frac{n+1}{n-1}-\frac{8(n-1)}{n(n-2)}\bigg{)}\|\na|\mathcal{W}_{g}|\|^{2}_{L^{2}}\\
		\geq &\bigg{(}\frac{2}{n}-\frac{2(n-1)}{\la(g)}C(N,k)\|\mathcal{Z}_{g}\|_{L^{ \frac{n}{2} }}\bigg{)}\|\mathcal{W}_{g}\|^{2}_{L^{\frac{2n}{n-2} }}+\bigg{(}\frac{n+1}{n-1}-\frac{8(n-1)}{n(n-2)}\bigg{)}\|\na|\mathcal{W}|_{g}\|^{2}_{L^{2}}\\
		\end{split}
		\end{equation*}
		where the last inequality comes from the repeated use of the conformal  Sobolev inequality \eqref{confSobo} together with the  curvature condition  $$ \bigg{(}\int_{M}(\sqrt{\frac{ 1}{n-2}}|\overset{\circ}{Ric}_{g}|+|\mathcal{W}_{g}|)^{\frac{n}{2}}\bigg{)}^{\frac{2}{n} }<c_{n}\la(g)$$
		with
		$$c(n)=\frac{1}{\sqrt{2n}(n-1)},$$
			Noting that
		\begin{equation*}
		\frac{n+1}{n-1}-\frac{8(n-1)}{n(n-2)}=\frac{n(n-2)(n-7)-8 }{n(n-1)(n-2)} .
		\end{equation*}	
		For any $n\geq8$, we have
		\begin{equation*}
		\frac{n+1}{n-1}>\frac{8(n-1)}{n(n-2)}.
		\end{equation*}	
		Therefore, we conclude that the coefficients of the terms $\|\mathcal{W}_{g}\|^{2}_{L^{\frac{2n}{n-2} } }$ and $\|\na |\mathcal{W}_{g}|\|^{2}_{L^{2}}$ are  strictly positive. 	Consequently, $\mathcal{W}_g \equiv 0$.

\end{proof}

\section*{Acknowledgements}
This work is supported by the National Natural Science Foundation of China No. 12271496 (Huang) and the Youth Innovation Promotion Association CAS, the Fundamental Research Funds of the Central Universities, the USTC Research Funds of the Double First-Class Initiative. 

\noindent\textbf{Data availability} {This manuscript has no associated data.}
\section*{Declarations}
\noindent\textbf{Conflict of interest} The author states that there is no conflict of interest.

\end{document}